\documentclass[12pt]{amsart}
\usepackage{amssymb,amsthm,amscd}
\usepackage[pagebackref]{hyperref}

\newcommand{\C}{{\mathbb C}}
\newcommand{\Q}{{\mathbb Q}} 
\newcommand{\Z}{{\mathbb Z}}
\newcommand{\FF}{{\mathbb F}}

\newcommand{\Fts}{F^{\times\,2}}
\newcommand{\fa}{{\mathfrak a}}

\newcommand{\fd}{{\mathfrak d}}

\newcommand{\frf}{{\mathfrak f}}

\newcommand{\fl}{{\mathfrak l}}
\newcommand{\fm}{{\mathfrak m}}
\newcommand{\fn}{{\mathfrak n}}
\newcommand{\fp}{{\mathfrak p}}
\newcommand{\fq}{{\mathfrak q}}

\newcommand{\cO}{{\mathcal O}}
\newcommand{\gen}{{\operatorname{gen}}}
\newcommand{\sep}{{\operatorname{sep}}} 
\newcommand{\Hom}{\operatorname{Hom}} 
\newcommand{\Disc}{{\operatorname{Disc}\,}}
\newcommand{\Dir}{{\operatorname{Dir}}}
\newcommand{\Kr}{{\operatorname{Kr}}}
\newcommand{\rk}{{\operatorname{rk}\,}}

\newcommand{\Ram}{\operatorname{Ram}}

\newcommand{\Cl}{\operatorname{Cl}}

\newcommand{\SCl}{\mbox{\rm SCl}}
\newcommand{\Sep}{\mbox{\rm Sep}}
\newcommand{\disc}{\operatorname{disc}\,}

\newcommand{\im}{\operatorname{im}}
\newcommand{\Sel}{\operatorname{Sel}}
\newcommand{\eps}{\varepsilon}
\newcommand{\impl}{\Longrightarrow}
\newcommand{\hra}{\hookrightarrow}
\newcommand{\lra}{\longrightarrow}

\newcommand{\la}{\langle}
\newcommand{\ra}{\rangle}

\newcommand{\ov}{\overline}
\newcommand{\biota}{\ov{\iota}}
\newcommand{\bpi}{\ov{\pi}}

\def\rsp{\raisebox{0em}[2.6ex][1ex]{\rule{0em}{2ex}}}

\newcounter{lemmacount}[section]

\newtheorem{thm}[lemmacount]{Theorem}
\newtheorem{prop}[lemmacount]{Proposition}
\newtheorem{lem}[lemmacount]{Lemma}
\newtheorem{cor}[lemmacount]{Corollary}

{}
{\bf}{\it}
{\bf}{}

\title{Dirichlet's Lemma in Number Fields}
\author{Franz Lemmermeyer}
\address{M\"orikeweg 1, 73489 Jagstzell, Germany}
\email{franz.lemmermeyer@gmx.de}

\begin{document}

\begin{abstract}
Dirichlet's Lemma states that every primitive quadratic Dirichlet 
character $\chi$ can be written in the form 
$\chi(n) = (\frac{\Delta}n)$ for a suitable quadratic discriminant
$\Delta$. In this article we define a group, the separant class
group $\SCl(F)$, that measures the extent to which Dirichlet's Lemma 
fails in general number fields $F$. As an application we will
show that over fields with trivial separant class groups, genus theory
of quadratic extensions can be made as explicit as over the rationals.
\end{abstract}

\maketitle

\begin{center}
{\large Dedicated to the memory of Peter Roquette (1927--2023)}
\end{center}

\medskip

In this article we will introduce the separant class group $\SCl(F)$
of a number field $F$ as a finite $2$-group measuring the obstruction
of Dirichlet characters to be induced by Kronecker symbols. We compute
its order and show that the separant class group is trivial if and only if
$F$ is totally real with odd class number in the strict sense.
In fields with trivial separant class groups, there is an analogue
of unique factorization into prime discriminants, and genus theory
can be made as explicit as over the rationals.

\section{Genus Theory of Quadratic Number Fields}

We begin by reviewing a few classical results concerning quadratic
number fields, the factorization of discriminants into prime discriminants,
and connections with genus theory.

Let $\Delta$ denote the discriminant of a quadratic number field $K$.
Quadratic discriminants (together with the discriminant $1$ of the trivial
extension $\Q/\Q$) form a group $\Disc(\Q)$ with respect to the following
multiplication: if $\Delta_1$ and $\Delta_2$ are two quadratic
discriminants, then  $\Delta_1 * \Delta_2 = \Delta_3$, where $\Delta_3$
is the discriminant of the extension $\Q(\sqrt{\Delta_1\Delta_2}\,)$.
The neutral element is the discriminant $1$ of the trivial
extension $\Q/\Q$. Observe that $\Delta * \Delta = 1$
(the group $\Disc(\Q)$ is an elementary abelian $2$-group) and that,
e.g., $(-8) * (-4) = 8$. If $\Delta_1$ and $\Delta_2$ are
coprime, then clearly $\Delta_1 * \Delta_2 = \Delta_1\Delta_2$.

Prime discriminants are discriminants having only one prime factor, namely
\begin{itemize}
\item the discriminants $-4$, $\pm 8$ of $\Q(\sqrt{-1}\,)$ and
  $\Q(\sqrt{\pm 2}\,)$;
\item the discriminant $p$ of $\Q(\sqrt{p}\,)$ for primes $p \equiv 1 \bmod 4$;
\item the discriminant $-q$ of $\Q(\sqrt{-q}\,)$
  for primes $q \equiv 3 \bmod 4$.
\end{itemize}
It is well known and easy to show that every discriminant $\Delta$ can
be written uniquely as a product of prime discriminants:
$$ \Delta = \Delta_1 \cdots \Delta_r \Delta_{r+1} \cdots \Delta_t, $$
where we assume that $\Delta_j < 0$ for $j \le r$ and $\Delta_j > 0$ for
$j > r$. In this case, 
$$ K_\gen^+ = \Q(\sqrt{\Delta_1}, \ldots, \sqrt{\Delta_t}\,) $$
is the genus class field of $K$ in the strict sense, that is, the
maximal abelian extension (necessarily an elementary abelian
$2$-extension) $L/\Q$ such that $L/K$ is unramified at all finite
places. The ordinary genus field $K_\gen$ of $K$, i.e., the maximal
unramified extension of $K$ that is abelian over $\Q$, is then given by
$$ K_\gen = \Q(\sqrt{\Delta_1\Delta_2}, \ldots, \sqrt{\Delta_1\Delta_r},
\sqrt{\Delta_{r+1}}, \ldots, \sqrt{\Delta_t}\,) $$
  if $\Delta > 0$ (for negative discriminants, $K_\gen = K_\gen^+$).
Observe that the genus field $K_\gen$ in the usual sense is the maximal
real subfield of $K_\gen^+$.
  
\medskip\noindent{\bf Example 1.}
The quadratic number field $K = \Q(\sqrt{210}\,)$ has discriminant
$\Delta = 840$, and the factorization of $\Delta$ into prime discriminants is
$$ \Delta = 8 \cdot (-3) \cdot 5 \cdot (-7), $$
hence the genus field of $K$ in the strict and in the usual sense is given 
respectively by
$$ K_\gen^+ = \Q(\sqrt{2}, \sqrt{-3}, \sqrt{5}, \sqrt{-7}\,)
\quad \text{and} \quad
    K_\gen = \Q(\sqrt{2}, \sqrt{5}, \sqrt{21}\,). $$

A quadratic number field has odd class number (in the usual sense)
if and only if $K_\gen = K$. This condition is obviously necessary;
for sufficiency observe that if $K$ has even class number, then its
Hilbert $2$-class field can be constructed in central steps, i.e., there
exists a central unramified quadratic extension of $K$, which necessarily
is elementary abelian over $\Q$ and thus is contained in the genus class field
of $K$ (for details, see \cite{Fro}). The explicit description of the genus
field then immediately implies the following classification of quadratic
number fields with odd class number:

\begin{thm}
The quadratic extensions $\Q(\sqrt{m}\,)$ with odd class number are
      \begin{itemize}
      \item the complex quadratic number fields
        $\Q(\sqrt{-1}\,)$, $\Q(\sqrt{-2}\,)$, and $\Q(\sqrt{-q}\,)$
            for primes $q \equiv 3 \bmod 4$;
      \item the real quadratic number fields $\Q(\sqrt{p}\,)$ for
            arbitrary primes $p$, and the fields $\Q(\sqrt{pq}\,)$
            for primes $p$ and $q$, where $p = 2$ or $p \equiv 3 \bmod 4$
            and $q \equiv 3 \bmod 4$.
      \end{itemize}
\end{thm}

\subsection*{Genus characters}
Given an ideal class $c \in \Cl^+(K)$, choose an ideal $\fa \in c$ coprime to
$\Delta$. Then $\chi_j(c) = (\frac{\Delta_j}{N\fa})$ defines a quadratic
character on the ideal class group $\Cl^+(K)$ called a genus character.
The main result concerning genus characters is the

\begin{thm}[Principal Genus Theorem]
  An ideal class $c \in \Cl^+(K)$ is a square if and only if
  $\chi_j(c) = 1$ for $j = 1, \ldots, t$.
\end{thm}

In fact, the homomorphism $\Cl^+(K) \lra \mu_2^t$ (where $\mu_2 = \{-1, +1\}$)
induced by $X = (\chi_1, \ldots, \chi_t)$ gives rise to an exact sequence
$$ \begin{CD}
  1 @>>> \Cl^+(K)^2 @>>> \Cl^+(K) @>{X}>> \mu_2^t @>>> \mu_2 @>>> 1.
  \end{CD} $$

Work by Goldstein \cite{Gold}
and Sunley \cite{Sun72,Sun79} (see also Davis \cite{Dav78}) shows that,
for a base field $F$, unique factorization into prime discriminants holds
if and only if $F$ is totally real and has class number $1$ in the strict
sense. Replacing discriminants by separants we will be able to generalize
these results to fields with trivial separant class group, i.e., to totally
real fields with odd class number in the strict sense (see Theorem \ref{TM2}). 
The notion of separants goes back to the author's dissertation \cite{Lem95}, 
where it was used to generalize R\'edei's construction of unramified
cyclic quartic extensions.

The construction of cyclic quartic unramified extensions of
$\Q(\sqrt{\Delta}\,)$  is also based on the factorization of the
discriminant $\Delta$ into prime discriminants:

\begin{thm}
  Let $\Delta$ be the discriminant of a quadratic number field $K$. Then
  the following assertions are equivalent:
  \begin{enumerate}
  \item[(1)] $K$ admits a cyclic quartic extension $L/K$ containing
    $\Q(\sqrt{\Delta_1},\sqrt{\Delta_2}\,)$ and unramified at all finite
    primes.
  \item[(2)]  There exists a factorization $\Delta = \Delta_1\Delta_2$
    into coprime discriminants $\Delta_1$ and $\Delta_2$, and the primes
    dividing $\Delta_1$ split completely in $\Q(\sqrt{\Delta_2}\,)$ and
    vice versa.
  \end{enumerate}
\end{thm}

Our results will allow us to generalize this construction to totally
real base fields $F$ with odd class number in the strict sense.

\section{Dirichlet's Lemma in $\Q$}

Let us now look at Kronecker symbols  $(\frac{\Delta}p)$, which
describe the splitting of primes in the quadratic number field with
discriminant $\Delta$, and quadratic Dirichlet characters defined modulo $m$.
Given a quadratic extension $K/\Q$ with discriminant $\Delta$, we define the
Kronecker symbol by setting
$$ \Big(\frac{\Delta}p\Big) = \begin{cases}
    +1 & \text{ if $p$ splits in $K/\Q$},      \\
      0  & \text{ if $p$ ramifies in $K/\Q$},   \\
     -1 & \text{ if $p$ is inert in   $K/\Q$}.
\end{cases} $$
Observe that the Kronecker symbol only depends on the ideal generated by $p$
in the sense that $(\frac{\Delta}p) = (\frac{\Delta}{-p})$.

Quadratic Dirichlet characters, on the other hand, are homomorphisms from
$(\Z/m\Z)^\times$  to the (multiplicatively written) group
$\mu_2 \simeq \Z/2\Z$. We identify Dirichlet characters modulo $m$
and modulo $n$ if they agree on all integers coprime to $mn$.

A quadratic Dirichlet character defined modulo $m$ is called primitive if
it cannot be defined for any proper divisor of $n$. The three nontrivial
quadratic Dirichlet characters defined modulo $8$ are
$$ \begin{array}{c|rrrr}
   \rsp          & 1 & 3 & 5 & 7 \\ \hline 
   \rsp \chi_8   & +1 & -1 & -1 & +1 \\ 
   \rsp \chi_{-8} & +1 & +1 & -1 & -1 \\
   \rsp \chi_{-4} & +1 & -1 & +1 & -1  
\end{array} $$
Their values at odd positive integers agree with $(\frac2p)$, $(\frac{-2}p)$
and $(\frac{-1}{p})$, respectively. The characters $\chi_8$ and
$\chi_{-8}$ are primitive, the character $\chi_{-4}$ is not as it is already
defined modulo $4$.

The smallest defining modulus of a character is called its conductor;
for Kronecker characters $\kappa_a = (\frac{a}{\cdot})$, the conductor
$N = |\Delta|$ agrees with the absolute value of the discriminant
$\Delta$ of the quadratic number field $\Q(\sqrt{a}\,)$.

The fact that Kronecker symbols $(\frac{\Delta}p)$ for primes $p > 0$
only depend on the residue class of $p$ modulo~$\Delta$ is essentially
equivalent to the quadratic reciprocity law. This observation assigns a
primitive quadratic Dirichlet character to each Kronecker symbol, and
Dirichlet's Lemma states that every primitive quadratic Dirichlet
character arises in this way (see below for a precise statement).

The same construction works for any algebraic number field, but it
turns out that, in general, not every primitive quadratic Dirichlet
character comes from a Kronecker symbol. For measuring this defect
we introduce the separant class group $\SCl(F)$ of a number field $F$,
and we will determine its structure.

We start by recalling Dirichlet's Lemma\footnote{The name ``Dirichlet's
  Lemma'' was coined by Cohn \cite{Coh}.}, which goes back to Dirichlet's
proof that there are infinitely many primes in arithmetic progressions.
In this proof, Dirichlet invented the characters that now bear his
name (strictly speaking Dirichlet provided the tools, and Dedekind
came up with the abstract concept of a character when he edited
Dirichlet's lectures). Dirichlet characters are group homomorphisms
$\chi: (\Z/m\Z)^\times \lra \C^\times$. Dirichlet's main problem was showing
that L-series
$$ L(s,\chi) = \sum_{n=1}^\infty \chi(n) n^{-s} $$
associated to Dirichlet characters $\chi$ have the property 
\begin{equation}\label{EDirL}
 \lim_{s \to 1} L(s,\chi) \ne 0
\end{equation}
for all characters $\chi$ different from the principal character.
This was not difficult except for quadratic characters $\chi$, i.e.,
characters that attain only the values $+1$
and $-1$. Eventually, Dirichlet realized that the limit in (\ref{EDirL})
for quadratic characters $\chi$ could be evaluated directly by showing
that every primitive quadratic Dirichlet character $\chi$ has the form
$\chi(n) = (\frac{\Delta}n)$ for some quadratic discriminant $\Delta$,
which allowed him to show that the limit in question is a product of
nonzero constants and the class number and the regulator of the quadratic
number field $\Q(\sqrt{D}\,)$. Again this is only true cum grano salis,
as Dirichlet worked with quadratic forms instead of fields. 

A Kronecker character $\kappa = (\frac{\Delta}n)$ is called
{\em modular} if there exists a modulus $N > 0$ such that
$\kappa(n) = \kappa(n+N)$ for all natural numbers $n$. It is easy to
see that if $\kappa$ is modular with respect to $N_1$ and $N_2$, then
it is modular with respect to $\gcd(N_1, N_2)$. The minimal modulus
is called the conductor of $\kappa$.

Similarly, a quadratic number field $K$ is called {\em modular} if
there exists an integer $N > 0$ such that $K$ is contained in the field
$\Q(\zeta_N)$ of $N$-th roots of unity. Again it is easy to see that if
$K$ is modular with respect to $N_1$ and $N_2$, then it is modular with
respect to $\gcd(N_1, N_2)$. The minimal modulus for $K$ is called the
conductor of $K$.

The similarity between these two notions of modularity is not superficial:
the splitting of primes $p$ in quadratic number fields $K$ with discriminant
$\Delta$ depends on the Kronecker character $\kappa(p) = (\frac{\Delta}{p})$,
that in cyclotomic fields $\Q(\zeta_N)$ on the values $\chi(p)$ for certain
Dirichlet characters $\chi$ defined modulo $N$. The modularity of $K$
can be shown to imply the modularity of $\kappa$, and in fact the
conductor of the Kronecker character $(\frac{\Delta}{\cdot})$ is equal to
the conductor $N = |\Delta|$ of the corresponding quadratic number field $K$.

\subsection*{Infinite Kronecker characters}
The Kronecker symbol $\kappa_4(n) = (\frac{-4}n)$ is a Dirichlet
character defined modulo $4$ on positive integers $n$ since
$$ \kappa_4(n) = \begin{cases}
                +1 & \text{ if } n \equiv 1 \bmod 4, \\
                 0 & \text{ if } n \equiv 0 \bmod 2, \\  
                -1 & \text{ if } n \equiv 3 \bmod 4.
                \end{cases} $$
If we think of the Kronecker symbol $\kappa_4(n)$ as a function of the
ideal $(n)$ generated by $n$, then we have to set $\kappa_4(-n) = \kappa_4(n)$.
But then $1 \equiv -3 \bmod 4$, yet we have $\kappa_4(1) = 1$ and
$\kappa_4(-3) = \kappa_4(3) = -1$. To get around this problem we
either consider our characters only on positive integers (and, over
number fields, on totally positive algebraic integers), or we
introduce an infinite Kronecker symbol 
$$ \kappa_\infty(n) = \Big(\frac{\infty}{n}\Big) = 
             \begin{cases} +1 & \text{ if } n > 0, \\
                           -1 & \text{ if } n < 0. \end{cases} $$
Then $\kappa_4(n) \kappa_\infty(n) = \chi(n)$ for all nonzero integers 
$n$, where $\chi$ is the nontrivial Dirichlet character modulo $4$.

%Infinite Kronecker characters are necessary for turning a Dirichlet
%character $\chi$ into a character $\bchi$ depending only on the ideal;
%this is achieved by setting $\bchi((a)) = \chi(a \kappa_\infty(a))$.

Dirichlet showed that every Dirichlet character is a Kronecker
character; in fact, we even have 

\medskip\noindent{\bf Dirichlet's Lemma.}
{\em There is a bijection between primitive quadratic Dirichlet 
     characters modulo $N$ and Kronecker characters with conductor $N$.}
\medskip

The fact that, conversely, Kronecker characters $(\frac{\Delta}{\cdot})$ 
are Dirichlet characters defined modulo $|\Delta|$ is equivalent to
Euler's version of the quadratic reciprocity law (see \cite{Baum},
as well as, for a proof of Dirichlet's Lemma, \cite{LemQNF}).

\section{Dirichlet's Lemma in Number Fields}

We now define the groups of Kronecker characters $\Kr(F)$ and the
group of quadratic Dirichlet characters $\Dir(F)$ in arbitrary
number fields $F$; then we will show that each Kronecker character
induces a Dirichlet character, and define the separant class group
as the quotient group $\Dir(F)/\Kr(F)$.

\subsection*{Dirichlet and Kronecker characters}
A quadratic Dirichlet character in $F$ defined modulo an ideal
$\fm$ in $\cO_F$ is a homomorphism from the
coprime residue class group $(\cO_F/\fm)^\times$ to the group
$\mu_2 = \{-1, +1\}$. We identify Dirichlet characters that are defined
modulo $\fm$ and $\fn$ and agree on elements coprime to $\fm\fn$.

Each element $\alpha \in F^\times$ defines a Kronecker character
$\kappa_\alpha = (\frac{\alpha}{\cdot})$.  As in the special case
$F = \Q$, $\kappa$ attains the values $+1$, $0$ and $-1$
according as $\fp$ splits, ramifies or is inert in $K/F$, where
$K = F(\sqrt{\alpha}\,)$.

Since we want to compare Kronecker and Dirichlet characters, and since
Dirichlet characters are defined on elements of the number field, we
identify two Kronecker characters $\kappa_\alpha$ and $\kappa_\beta$
if they have the same values at all elements coprime to $\alpha$ and
$\beta$; for example, we identify the Kronecker symbols
$(\frac{-1}{\cdot})$ and $(\frac{-4}{\cdot})$ over $\Q$.

A less trivial example is the Kronecker character
$\kappa = (\frac{-1}{\cdot})$ in $F = \Q(\sqrt{-5}\,)$. We have
$\kappa(\fp) = -1$ for the prime ideal $\fp = (3,1+\sqrt{-5}\,)$
since this ideal is inert in the Hilbert class field $F(i)$ of $F$.
But $\kappa$ is trivial on principal ideals by the decomposition
theorem of class field theory, and in particular we identify
$\kappa$ with the trivial character.
Thus Kronecker characters are the restrictions of Artin symbols of
quadratic extensions $K/F$ to principal ideals.

We will evaluate Kronecker characters  $(\frac{\alpha}{\beta})$ only
at integers $\beta$ coprime to the relative discriminant $\disc(K/F)$,
where $K = F(\sqrt{\alpha}\,)$. For example, the Kronecker symbol
$(\frac{3}{m})$ is defined only for integers coprime to $6$.

Kronecker characters form a group $\Kr(F)$ with respect to the product
$\kappa_\alpha \kappa_\beta = \kappa_{\alpha\beta}$. If $\alpha \in F^\times$
is a square, then $\kappa_\alpha$ is the trivial character. This observation
implies that there is a homomorphism
\begin{equation}\label{Dphi}
  \phi: F^\times/F^{\times\,2} \lra \Kr(F)
\end{equation}
defined by $\phi(\alpha F^{\times\,2}) = \kappa_\alpha$.
Different elements in $F^\times/F^{\times\,2}$ do not necessarily
give rise to different Kronecker symbols; in other words: $\phi$
is not necessarily injective.

Recall that we have introduced the following groups in
\cite{LemSel}:
\begin{align*}
  \Sel(F)   & = \{\alpha \in F^\times: (\alpha) = \fa^2\}/F^{\times\,2}, \\
  \Sel_4(F) & = \{\alpha \in F^\times: (\alpha, 2) = (1),
  (\alpha) = \fa^2, \alpha \equiv \xi^2 \bmod 4\} \cdot F^{\times\,2}/F^{\times\,2}
\end{align*}
Thus $\Sel(F)$ is the group of all nonzero elements of $F$ that generate
the square of an ideal (modulo squares of elements), and $\Sel_4(F)$ is its 
subgroup of elements with odd norm that generate the square of an ideal and
are congruent to a square modulo $4$. Elements with odd norm congruent to
a square modulo $4$ are called $2$-primary; the quadratic extensions
$F(\sqrt{\alpha}\,)/F$ are unramified at $2$ when $\alpha$ is $2$-primary.

The next proposition determines the kernel of the homomorphism $\phi$
defined in (\ref{Dphi}):

\begin{prop}
  We have an exact sequence
  $$ \begin{CD}
    1 @>>> \Sel_4(F) @>{\iota}>> F^\times/F^{\times\,2}
                          @>{\phi}>> \Kr(F) @>>> 1. \end{CD} $$
\end{prop}

Observe that if $\alpha \in \Sel_4(F)$ is nontrivial, then
$F(\sqrt{\alpha}\,)/F$ is a quadratic extension unramified at all the
finite primes; in particular, the class number of $F$ in the strict sense
is even in this case.

If the class number of $F$ is odd, then $\Sel_4(F) \simeq E_4/E^2$,
where $E_4$ is the group of units congruent to a square modulo $4$.
An explicit example of a unit in
$E_4 \setminus E^2$ is given by the fundamental unit $\eps = 35 + 6 \sqrt{34}$
of $\Q(\sqrt{34}\,)$; since $2 \eps = (6 + \sqrt{34}\,)^2$ we have
$F(\sqrt{\eps}\,) = F(\sqrt{2}\,)$, and this is an unramified quadratic
extension of $F$.

\begin{proof}
  We only have to show that $\ker \phi = \im \iota$. Assume that
  $\kappa \in \ker \phi$; then $\kappa = (\frac{\alpha}{\cdot})$ is the
  trivial Kronecker character. Since the conductor of the
  Kronecker symbol is the conductor of the quadratic extension
  $F(\sqrt{\alpha})/F$, this extension is unramified at all finite
  primes; but this implies that  $(\alpha)$ is a square of an ideal
  and that $\alpha = \beta \gamma^2$ for some primary $\beta$,
  i.e., that $\alpha \Fts \in \Sel_4(F)$.
\end{proof}

Next we show that Kronecker characters are Dirichlet characters:
\begin{prop}
  Each Kronecker character $\kappa_\alpha = (\frac{\alpha}{\cdot})$
  is a primitive quadratic Dirichlet character with conductor
  $\disc(K/F)$, where $K = F(\sqrt{\alpha}\,)$.
\end{prop}

\begin{proof}
  By Artin's reciprocity law, the splitting behavior of a prime ideal in
  a quadratic extension $K = F(\sqrt{\alpha}\,)$ only depends on the coset
  of $\fp$ in the corresponding ideal group, and the conductor of this
  group is the relative discriminant of $K/F$. In particular, for any
  $\beta \gg 0$ (i.e., $\beta$ is totally positive -- at this point we
  restrict our characters to totally positive integers instead of
  introducing infinite Kronecker characters) in $F^\times$ coprime
  to $2\alpha$ we have $(\frac{\alpha}{\beta}) = \chi(\beta)$ for
  some Dirichlet character with conductor $\disc(K/F)$. Since the
  conductor is the minimal defining modulus, $\chi$ is primitive.
\end{proof}

This result will allow us to show that certain Dirichlet characters
are not Kronecker characters: a necessary condition for the existence
of a Kronecker character with conductor $\frf$ is the existence of
a quadratic extension $K/F$ with relative discriminant
$\disc(K/F) = \frf$.

\begin{prop}
  If $F$ has odd class number in the strict sense, then Kronecker characters
  can be identified with separants of quadratic extensions $K/F$.
\end{prop}

It is not enough to assume that $F$ has odd class number in the usual sense.
In fact, consider the field $F = \Q(\sqrt{6})$ with fundamental unit
$\eps = 5 + 2\sqrt{6}$. The separant $-\eps E^2$ belongs to the quadratic
extension $K = F(\sqrt{-\eps}\,) = \Q(\sqrt{-2},\sqrt{-3}\,)$ since
$-\eps = (\sqrt{-2} + \sqrt{-3}\,)^2$, and $K/F$ is unramified at all
finite primes. This implies that principal ideals are in the kernel
of the Artin map, which in turn means that the Kronecker symbol
$(\frac{-\eps}{\cdot})$ is trivial on $F^\times$ although $\eps$ is
not a square.

\begin{proof}
  The natural homomorphism $\kappa: F^\times/F^{\times\,2} \lra \Kr(F)$
  defined by $\kappa(\alpha F^{\times\,2}) = (\frac{\alpha}{\cdot})$
  is clearly surjective. The kernel consists of all $\alpha F^{\times\,2}$
  such that  $\kappa(\alpha F^{\times\,2})$ has conductor $(1)$. This is
  only possible if $\alpha$ is (up to squares) a unit congruent to a
  square modulo $4$. Since $F$ has odd class number in the strict sense,
  the only such units are squares.

  The separant $\delta E^2$ corresponds to the Kronecker symbol
  $(\frac{\delta}{\cdot})$. The homomorphism $\lambda: \Sep(F) \lra \Kr(F)$
  with $\lambda(\delta E^2) = \kappa_\delta$ is injective by what  we have
  already shown. We claim that it is also surjective. In fact,
  let $\kappa_\alpha$ denote any Kronecker character, and let
  $\delta = \sep(F(\sqrt{\alpha}\,)/F)$. Then
  $\delta F^{\times\,2} = \alpha F^{\times\,2}$, and
  $\kappa_\alpha(\beta) = \kappa_\delta (\beta)$
  for all elements $\beta \in F^\times$ coprime to $\alpha\delta$.
\end{proof}

\section{Determination of the Separant Class Group}

The separant class group $\SCl(F) = \Dir(F)/\Kr(F)$ of a number field $F$
is by definition an elementary abelian $2$-group, and so it is sufficient
to determine its order, or its dimension as an $\FF_2$-vector space. This
is accomplished by our main theorem:

\begin{thm}\label{TM1}
  There is an isomorphism
  $$ \SCl(F) \simeq \Cl_F\{4\}/\Cl_F\{4\}^2, $$
  where $\Cl_F\{4\}$ is the ray class group of $F$ defined modulo $(4)$.
  In particular we have $\#\SCl(F) = 2^{\rho^+ + s}$, where
  $2^{\rho^+} = \# \Cl_2^+(K)/\Cl_2^+(K)^2$ denotes the $2$-part of the
  class group of $F$ in the strict sense, and $s$ the number of pairs
  of complex embeddings of $F$. Thus $F$ has trivial separant class group
  if and only if $F$ is totally real with odd class number in the strict
  sense.
\end{thm}

%The fact that  $\#\SCl(F) = 2^{\rho^+ + s}$ follows from
%$\# \Cl_F\{4\}/\Cl_F\{4\}^2 = 2^{\rho^++s}$, which was proved in \cite{LemSel}.

We will prove this theorem in two steps: First we discuss the case of
characters defined modulo $4$, and then the general case.

\subsection*{Dirichlet Characters with conductor dividing $(4)$}

Our proof uses a few facts concerning Kronecker and Dirichlet characters
whose conductors divide $(4)$. Let $\Kr_4(F)$ denote the group of
Kronecker characters  with conductor dividing $(4)$, $\Dir_4(F)$
the group of Dirichlet characters defined modulo $(4)$, and
$M_4 = (\cO_F/4\cO_F)^\times$ the group of coprime residue classes modulo $4$.
For the proof of Thm.~\ref{TM1} we only need the fact that the homomorphism
$\iota_2: M_4/M_4^2 \lra \Dir_4(F)$ is surjective.

\begin{prop}\label{PD4}
  There is an isomorphism $\iota_2: M_4/M_4^2 \lra \Dir_4(F)$. Both groups
  have order $2^n$, where $n = (F:\Q)$.
\end{prop}

\begin{proof}
  Quadratic Dirichlet characters defined modulo $4$ are homomorphisms
  from $M_4$ to the multiplicatively written group $\Z/2\Z$. Since such
  homomorphisms are trivial on squares, they induce homomorphisms
  $M_4/M_4^2 \lra \Z/2\Z$. Thus $\Dir_4(F) = \Hom(M_4/M_4^2, \Z/2\Z)$,
  which means that $\Dir_4(F)$ is the dual of $M_4/M_4^2$ (for example
  as $\Z/2\Z$-vector spaces); in particular, $\Dir_4(F) \simeq M_4/M_4^2$
  as abelian groups.

  The isomorphism between $M_4/M_4^2$ and $\Dir_4(F)$ given above is not
  canonical, but depends on the choice of a basis. In this case, however,
  there is also a canonical isomorphism that we will now construct.
  
  Let $\alpha$ represent a coprime residue class modulo $4$, and let
  $\kappa_\alpha$ denote the corresponding Kronecker character.
  Write $\kappa_\alpha = \chi_0 \chi_1$ as the product of a primitive
  quadratic Dirichlet character $\chi_0$ defined modulo $4$ and
  a Dirichlet character $\chi_1$ defined modulo an odd ideal. Then set
  $\iota_2(\alpha M_4^2) = \chi_0$. 

  The map $\iota_2$ is well defined: If $\alpha \equiv \beta \bmod 4$,
  $\kappa_\alpha = \chi_0 \chi_1$ and $\kappa_\beta = \chi_0' \chi_1'$,
  then $\kappa_\alpha \kappa_\beta = \kappa_{\alpha\beta}$ is defined modulo
  an odd ideal since $\alpha/\beta \equiv 1 \bmod 4$. Thus $\chi_0 = \chi_0'$.

%  We observe that if $\alpha$ is a unit,
%  then $\iota_2(\alpha) = \kappa_\alpha$; thus the upper left square of
%  the diagram is commutative.

  We claim that $\iota_2$ is injective. Assume that
  $\iota_2(\alpha M_4^2) = \iota_2(\beta M_4^2)$. Then $\kappa_{\alpha\beta}$
  has conductor coprime to $2$, hence $\alpha\beta \equiv \xi^2 \bmod 4$,
  which implies that $\alpha$ and $\beta$ generate the same class in
  $M_4/M_4^2$.

  Since $M_4/M_4^2 \simeq (\Z/2\Z)^n$ for $n = (F:\Q)$ (see
  \cite[p. 281]{LemSel}; sending $\alpha \bmod 2$ to
  $1+2\alpha \bmod 4$) induces an isomorphism
  $\cO/2\cO \simeq (\Z/2\Z)^n \lra M_4/M_4^2$.
\end{proof}

%forstep(d=3,100,4,print(d,"   ",quadunit(4*d)))

\subsection*{Kronecker characters with conductor dividing $(4)$}

Next we determine the number of Kronecker characters with conductor
dividing $(4)$. 

\begin{prop}
  We have an exact sequence 
  $$ \begin{CD}
    1 @>>>  \Sel_4(F) @>{\iota}>>  \Sel(F)  @>{\nu}>>  \Kr_4(F) @>>>  1
  \end{CD} $$
\end{prop}

\begin{proof}
The homomorphism $\iota$ is an injection since $\Sel_4(F) \subseteq \Sel(F)$.
We set $\nu(\alpha F^{\times\,2}) = \kappa_{\alpha}$. Then $\ker \nu$
consists of the cosets $\alpha F^{\times\,2}$ such that $\kappa_\alpha = 1$;
this holds if and only if $\alpha \equiv \xi^2 \bmod 4$, i.e., if
$\alpha F^{\times\,2} \in \Sel_4(F)$.

Finally assume that $\kappa_\alpha \in \Kr_4(F)$. Then $(\alpha) = \fa^2$,
hence $\kappa_\alpha \in \im \nu$.
\end{proof}

We know from \cite{LemSel} that $\# \Sel_4(F) = 2^{\rho^+}$,
where $\rho^+$ denotes the $2$-rank of $\Cl_2^+(F)$, and that
$\# \Sel(F) = 2^{\rho + r + s}$, where $\rho$ denotes the $2$-rank of $\Cl_2(F)$,
and $r$ and $2s$ denote the number of real and complex embeddings of $F$.

\begin{cor}
  We have $\# \Kr_4(F) = 2^{r+s-(\rho^+-\rho)}$.
\end{cor}

\subsection*{The separant class group modulo $4$}

Define the separant class group modulo $4$ as the quotient
$$ \SCl_4(F) = \Dir_4(F) / \Kr_4(F). $$
Then
$$ \# \SCl_4(F) = 2^{n-r-s-\rho+\rho^+} = 2^{s+\rho^+ - \rho}. $$
This implies

\begin{cor}\label{CK4}
  Every Dirichlet character defined modulo $4$ is represented by
  a Kronecker character if and only if $F$ is totally real and
  $\Cl_2(F)$ and $\Cl_2^+(F)$ have the same rank.
\end{cor}

In fact we have $s \ge 0$ and $\rho^+ \ge \rho$, hence
$2^{s+\rho^+ - \rho} = 1$ if and only if $s = 0$ and $\rho^+ = \rho$.

\begin{cor}
  If $F$ has odd class number, then every Kronecker character with
  conductor dividing $4$ has the form $\kappa_\eps$ for a unit $\eps \in E_F$.
\end{cor}

\begin{proof}
  If $F$ has odd class number, then $\Sel(F) \simeq E/E^2$ and
  $\Sel_4(F) \simeq E_4/E^2$ by \cite[Prop. 3.1]{LemSel}. Thus we get
  the exact sequence
  $$ \begin{CD}
    1 @>>>  E_4/E^2 @>{\iota}>>  E/E^2  @>{\nu}>>  \Kr_4(F) @>>>  1,
  \end{CD} $$
  which implies the claim.
\end{proof}

In Table \ref{Tab1}, the column $\Kr_4(F)$ lists the elements $\alpha$
for which $\kappa_\alpha$ generates $\Kr_4(F)$.

\begin{table}[ht!]
$$ \begin{array}{l|ccc|cc}
  \rsp  F & r+s & \rho & \rho^+ & \Kr_4(F) & \SCl_4(F) \\ \hline
  \Q    & 1 & 0 & 0           &  -1               & 1 \\
  \Q(i) & 1 & 0 & 0           &   i               & \la \chi_{2i} \ra  \\
  \Q(\sqrt{-5}\,) & 1 & 1 & 1 &  -1               & \la \chi_2 \ra  \\
  \Q(\sqrt{2}\,)  & 2 & 0 & 0 & -1,\ 1+\sqrt{2}   & 1 \\
  \Q(\sqrt{3}\,)  & 2 & 0 & 1 & 2+\sqrt{3}        & \la \chi_2 \ra  \\
  \Q(\sqrt{6}\,)  & 2 & 0 & 1 &  -1               & \la \chi_4  \ra \\
  \Q(\sqrt{10}\,) & 2 & 1 & 1 & -1,\ 3+\sqrt{10}  &   1 \\
  \Q(\sqrt{15}\,) & 2 & 1 & 2 & 4 + \sqrt{15}     & \la \chi_2 \ra  \\
  \Q(\sqrt{34}\,) & 2 & 1 & 1 & -1, 5 + \sqrt{34} &  1  \\  
  \end{array} $$
  \caption{Separant class groups modulo $4$.
    Here $\dim \Kr_4(F) = r+s-(\rho^+-\rho)$ and
    $\dim \SCl_4(F) = s + \rho^+ - \rho$.}\label{Tab1}
\end{table}

We next determine the conductor of the Kronecker character
$\kappa_{-1} = (\frac{-1}{\cdot})$:

\begin{prop}
  The Kronecker character $\kappa_2 = (\frac{-1}{\cdot})$ in
  $F = \Q(\sqrt{m}\,)$ is a Dirichlet character with conductor $\frf$
  given by
  $$ \frf = \begin{cases}
    4 & \text{ if } m \equiv 1 \bmod 4, \\
    2 & \text{ if } m \equiv 2 \bmod 4, \\
    1 & \text{ if } m = 3 \bmod 4.
    \end{cases} $$
\end{prop}

This follows easily from the conductor-discriminant formula, and it is
also easily verified directly. If $m \equiv 2 \bmod 4$ and
$\alpha = a + b\sqrt{m} \gg 0$, for example, we have
$$ \Big(\frac{-1}{\alpha}\Big) \equiv (-1)^{\frac{a^2 - mb^2 - 1}2}
   =  (-1)^{\frac{a^2 - 1}2} (-1)^{\frac{mb^2}2} = (-1)^{b} \bmod \alpha, $$ 
hence
$$ \Big(\frac{-1}{\alpha}\Big) = \begin{cases}
  +1 & \text{ if } b \equiv 0 \bmod 2,
       \text{ i.e., } \alpha \equiv 1 \bmod 2, \\
  -1 & \text{ if } b \equiv 1 \bmod 2, 
       \text{ i.e., } \alpha \equiv 1 + \sqrt{m} \bmod 2.
  \end{cases} $$

The Kronecker character $\kappa_\eps = (\frac{\eps}{\cdot}\,)$,
where $\eps > 0$ is the fundamental unit of $F$, is trivial if
and only if $F(\sqrt{\eps}\,)/F$ is unramified. This is the case
for $m = 34$, where $\sqrt{35 + 6\sqrt{34}} = 3\sqrt{2} + \sqrt{17}$.
For $m = 6$, the element $\sqrt{-\eps} = \sqrt{-2} + \sqrt{-3}$ generates a
quadratic extension unramified at all finite primes, hence
$\kappa_{-1} = \kappa_{\eps}$.

\subsection*{Ray class groups}

Before we give a proof of Theorem~\ref{TM1} we recall a few basic facts about
ray class groups. The ray class group modulo $\fm$ is defined by the exact
sequence
$$ \begin{CD}
  1  @>>> P_\fm/P_\fm^1  @>>> I_\fm/P_\fm^1 @>>> I_\fm/P_\fm @>>> 1 
  \end{CD} $$
Here $I_\fm$ is the group of fractional ideals coprime to $\fm$ and $P_\fm$
its subgroup of principal ideals in $I_\fm$; the ray $P_\fm^1$
consists of all ideals $(\alpha)$ generated by elements
$\alpha \equiv 1 \bmod \fm$ (congruence in the multiplicative sense).
Observe that $I_\fm/P_\fm \simeq \Cl(F)$ since every ideal class contains
ideals coprime to $\fm$, and that $I_\fm/P_\fm^1 \simeq \Cl_F\{\fm\}$
by definition. Below we will use the abbreviation $R_\fm = P_\fm/P_\fm^1$.

The connection with units is provided by the exact sequence
$$ \begin{CD}
   1 @>>> E/E_\fm @>>> (\cO_F/\fm)^\times @>>> R_\fm @>>> 1,
  \end{CD} $$
where $E_\fm$ is the group of units $\eps \equiv 1 \bmod \fm$.
These sequences imply the well known formula
$$ \# \Cl_F\{\fm\} = h \cdot \frac{\Phi(\fm)}{(E:E_\fm)}. $$

The following simple lemma will be used repeatedly:

\begin{lem}\label{LTen}
  If the sequence
  $$ \begin{CD}
    1 @>>> A @>{\iota}>> B @>{\pi}>> C @>>> 1
  \end{CD} $$
  of abelian groups is exact, then so is  
  $$ \begin{CD}
    1 @>>> A/A\cap B^2 @>{\biota}>> B/B^2 @>{\bpi}>> C/C^2 @>>> 1.
  \end{CD} $$
\end{lem}

\begin{proof}
  We assume that $\iota$ is the inclusion map. 
  We only have to show that $\biota$ is injective and that
  $\ker \bpi = \im \biota$.

  Now $\iota(a) \in B^2$ implies that $a \in A \cap B^2$, hence
  $\biota$ is injective.

  Next $\pi(b) \in C^2$ implies $\pi(b) = \pi(b_1)^2$ with $\pi(b_1) = c$,
  hence $b/b_1^2 \in \ker \pi = \im \iota$. Thus $b/b_1^2 = a$, hence
  $b = ab_1^2$ and $bB^2 = aB^2 \in \im \biota$.
\end{proof}

\begin{cor}
  The sequence
  $$ \begin{CD}\label{CD1}
    1 @>>> E/E_4 @>>>  M_4/M_4^2 @>>> R_4/R_4^2 @>>>  1
  \end{CD} $$
  is exact.
\end{cor}

This follows by applying Lemma \ref{LTen} to the exact sequence
$$ \begin{CD}
   1 @>>> E/E_4 @>>> M_4  @>>> R_4 @>>> 1
  \end{CD} $$
by observing that $\eps E_4$ is in $M_4^2$ if and only if $\eps \in E_4$.

\begin{cor}\label{CorRC}
  The sequence
  $$ \begin{CD}
    1 @>>> P_F /P_F \cap I_F^2 @>>> I_F/I_F^2 @>>> \Cl(F)/\Cl(F)^2  @>>> 1
  \end{CD} $$
  is exact.
\end{cor}

Here $I_F$ denotes the group of nonzero factional ideals and 
$P_F$ its subgroup of principal ideals. The exactness of the sequence
follows by applying Lemma \ref{LTen} to the exact sequence
$$ \begin{CD}
   1 @>>> P_F @>>> I_F  @>>> \Cl(F) @>>> 1
  \end{CD} $$
that defines the ideal class group.

\begin{cor}
  The sequence
  $$ 1 \lra P_4/P_4 \cap P_4^1 I_4^2  \lra \Cl_F\{4\}/\Cl_F\{4\}^4
               \lra \Cl(F)/\Cl(F)^2 \lra 1 $$
  is exact.
\end{cor}

This follows by applying Lemma \ref{LTen} to the exact sequence
$$ \begin{CD}
   1 @>>> R_4 @>>> \Cl_F\{4\}  @>>> \Cl(F) @>>> 1
  \end{CD} $$
that defines the ray class group modulo $4$.

\subsection{Proof of the Main Theorem}
The main step in the proof of Thm.~\ref{TM1} is verifying the exactness
of the first two rows of the commutative diagram in Figure~\ref{CDSCl}.
Recall that we know
\begin{align*}
  \Kr_4(F) & \simeq \Sel(F)/\Sel_4(F), & \# \Kr_4(F) & = 2^{r+s-(\rho^+-\rho)} \\
  \Dir_4(F) & \simeq M_4/M_4^2,         & \# \Dir_4(F) & = 2^n, \\
  \SCl_4(F) & \simeq P_4 \cap P_4^1 I_4^2, & \#  \SCl_4(F) & = 2^{s+\rho^+ - \rho}.
\end{align*}
The exactness of the fundamental diagram then implies
$$ \SCl(F) \simeq \Cl_F\{4\}/ \Cl_F\{4\}^2, $$
as well as
$$ \# \SCl(F) = \# \SCl_4(F) \cdot (\Cl(F):\Cl(F)^2) = 2^{\rho^+ + s}. $$

\begin{figure}[ht!]
$$ \begin{CD}
  @.  1     @.    1  @. 1 @. \\  
  @. @VVV  @VVV  @VVV @. \\
  1 @>>> \Kr_4(F) @>>> \Kr(F) @>>> P_F/P_F \cap I_F^2 @>>> 1 \\
  @. @VVV  @VVV  @VVV @. \\
  1  @>>> \Dir_4(F) @>{\iota_2}>> \Dir(F) @>{\pi_2}>> I_F/I_F^2 @>>> 1 \\
  @. @VVV  @VVV  @VVV @. \\
  1 @>>>  \SCl_4(F) @>>> \SCl(F) @>>> \Cl(F)/\Cl(F)^2 @>>> 1 \\
  @. @VVV  @VVV  @VVV @. \\  
  @.  1     @.    1  @. 1 
  \end{CD} $$

  \caption{Fundamental diagram for the determination of $\SCl(F)$.}\label{CDSCl}
\end{figure}

  \begin{figure}[ht!]
  $$ \begin{CD}
    1 @>>> \Kr_4(F) @>>> \Dir_4(F) @>>> \SCl_4(F) @>>> 1 \\
    @. @VVV  @VVV  @VVV @. \\
    1 @>>> \frac{\Sel(F)}{\Sel_4(F)} @>>> \frac{M_4}{M_4^2}
      @>>> \frac{P_4}{P_4 \cap P_4^1 I_4^2} @>>> 1
    \end{CD} $$

    \bigskip
    
    $$ \begin{CD}
    1 @>>> \SCl_4(F) @>>> \SCl(F) @>>> \Cl(F)/\Cl(F)^2 @>>> 1 \\
    @. @VVV  @VVV  @VVV @. \\
    1 @>>> \frac{P_4}{P_4 \cap P_4^1I_4^2} @>>> \frac{\Cl_F\{4\}}{\Cl_F\{4\}^2}
                               @>>> \frac{\Cl(F)}{\Cl(F)^2} @>>> 1
    \end{CD} $$    
    \caption{The sequences in the left column and the bottom row;
      vertical maps are isomorphisms.}\label{CDSCl2}
  \end{figure}

The vertical sequences in the left and middle column of Fig. (\ref{CDSCl})
are the definitions of the groups at the bottom; the exactness of the
sequence in the right column is Cor. \ref{CorRC}.
The snake lemma then implies that the sequence in
the third row is exact. This implies that $\SCl(F)$  has the same order
as $\Cl_F\{4\}/\Cl_F\{4\}^2$, and hence is isomorphic to this group since
both are elementary abelian.

\subsection*{Exactness of the top row}
Since $\Kr_4(F) \simeq \Sel(F)/\Sel_4(F)$ we have to show that the
sequence 
$$ \begin{CD}
  1 @>>> \Sel(F)/\Sel_4(F) @>>> \Kr(F) @>>> P_F /P_F \cap I_F^2  @>>> 1
\end{CD} $$
is exact.

Define $j: \Sel(F)/\Sel_4(F) \lra \Kr(F)$ as the map sending
$\alpha \Fts$ to the Kronecker character $\kappa_\alpha$. This map
is well defined since Kronecker characters $\kappa_\alpha$ with
$\alpha \in \Sel_4(F)$ are trivial. Let $\lambda$ denote the map
$\Kr(F) \lra P_F/P_F \cap I_F^2$ induced by sending $\kappa_\alpha$ to
$(\alpha)$. Clearly $\im j \subseteq \ker \lambda$.  Assume now that
$\kappa_\alpha \in \ker \lambda$; then $\alpha = \fa^2$ is the square
of an ideal, hence $\alpha \Fts \in \im j$. Since $\lambda$ is clearly
surjective, this proves our claim.
  
\subsection*{Exactness of the middle row}
Next we prove the exactness of
$$ \begin{CD}
      1  @>>> M_4/M_4^2 @>{\iota_2}>> \Dir(F) @>{\pi_2}>> I_F/I_F^2 @>>> 1.
   \end{CD} $$
The map $\iota_2$ is defined as in Prop.~\ref{PD4}; recall that $\iota_2$
is injective, that it makes the left upper square commutative, and that
$\im \iota_2 = \Dir_4(F)$.

Let $\chi$ be a primitive quadratic Dirichlet character with conductor
$\fm$; then we set $\pi_2(\chi) = \fm I_F^2$. This is a homomorphism:
If $\chi$ and $\psi$ are  primitive quadratic Dirichlet characters with
conductor $\fm$ and $\fn$, then $\chi\psi$ has conductor $\fm\fn/\fa^2$
for a suitable ideal $\fa$. This is easy to see for odd prime ideals
dividing the conductor; for prime ideals $\fl$ above $2$ it follows from
the fact that the conductor of $\chi$ is an even power if and only if
$\chi \in \Dir_4(F)$.

We now claim that $\pi_2$ is surjective. Clearly every prime ideal with
odd norm is the conductor of a primitive Dirichlet character. If $\fl$
is a prime ideal above $2$, choose an $\alpha$ divisible exactly by $\fl$,
and take the $\fl$-part of the Kronecker symbol $\kappa_\alpha$. Since
$I_F/I_F^2$ is generated by prime ideals, $\pi_2$ is onto.

If $\chi \in \ker \pi_2$, then the conductor of $\chi$ is the square of
an ideal. Since $\chi$ is primitive, this implies that $\chi$ is defined
modulo an ideal dividing $4$, hence it is in the image of $\iota_2$
by Prop.~\ref{PD4}.

If we splice together the sequences in the left column and the bottom row
we obtain with the exact sequence
$$ \Sel(F)/\Sel_4(F) \hra M_4/M_4^2 \lra \Cl_F\{4\}/\Cl_F\{4\}^2
     \twoheadrightarrow \Cl(F)/\Cl(F)^2 $$
already derived in \cite{LemSel}.

If we write $\chi$ as the product $\chi_e\chi_o$ of a character $\chi_e$
with even and $\chi_o$ with odd conductor and send $[\chi] \in \SCl(F)$
to the ideal class of the conductor of $\chi_o$, this seems to lead into
delicate technical calculations involving conductors that are powers of
prime ideals above $(2)$ (see Hasse \cite[Satz 16$_2$]{HasZB}).

\subsection*{Examples}
The examples in the following table are those where $\SCl(F)$ is strictly
larger than $\SCl_4(F)$. 

$$ \begin{array}{c|cccc}
  \rsp  F      & \Q(\sqrt{-5}\,) & \Q(\sqrt{10}\,) &
               \Q(\sqrt{15}\,) & \Q(\sqrt{34}\,) \\ \hline 
  \rsp \SCl(F) & \la \chi_2, \chi_3 \ra & \la \chi_3 \ra &
                 \la \chi_2, \chi_3 \ra & \la \chi_\fp \ra 
  \end{array} $$
The separant class groups defined modulo $4$ were already described above.
The difference between $\SCl_4(F)$ and $\SCl(F)$ comes from nontrivial
$2$-class groups.
\begin{itemize}
\item $F = \Q(\sqrt{-5}\,)$. We already know that there is no Kronecker
  character with conductor $2$. There is also no Kronecker character
  with conductor $\fp$, the two prime ideals above $3$, since
  there is no quadratic extension ramified exactly at $\fp$.
  The product of the Dirichlet characters $\chi_\fp$ and $\chi_{\fp'}$
  is the Kronecker character $(\frac{-3}{\cdot})$, hence these two
  characters are equivalent in $\SCl(F)$.
\item $F = \Q(\sqrt{10}\,)$. The prime ideals above $3$
  are not principal, so the Dirichlet characters $\chi_3$ with conductor
  $3$ are not Kronecker characters.  
\item $F = \Q(\sqrt{34}\,)$. The fundamental unit $\eps = 35 + 6 \sqrt{34}$
  generates $E_4/E^2$. The Dirichlet character with conductor $3$ is not
  a Kronecker character.
\end{itemize}

\section{Separants}

Relative discriminants $\disc(K/F)$ of quadratic extensions $K/F$ of
number fields contain information about the ramified primes. Separants,
which we will define below, contain more information: they are (cosets of)
algebraic numbers whose square roots generate the quadratic extensions
in question. In this section we will define separants for number fields
$F$ with odd class number, and we will prove their most basic properties.

For quadratic extensions $K/F$, we let $\Ram(K/F)$ denote the set
of ramified places of $F$, and $\Ram_f(K/F)$ its subset of finite
ramified places, i.e., of prime ideals in $F$ that are ramified in $K/F$.

\begin{prop}
  Let $K/F$ be a quadratic extension of number fields. If $F$ has
  odd class number and $K = F(\sqrt{\alpha}\,)$, then there
  exists an integer $\delta \in F^\times$ unique up to squares of units
  such that
  \begin{itemize}
  \item $F(\sqrt{\alpha}\,) =F(\sqrt{\delta}\,)$;
  \item $\disc(K/F)^h = (\delta)$.
  \end{itemize}
\end{prop}

\begin{proof}
  Assume that $F$ has odd class number $h$, and consider quadratic
  extensions $K = F(\sqrt{\alpha}\,)$. The relative discriminant
  $\fd = \disc(K/F)$ is an ideal, and we know that
  $4\mu = \disc(\sqrt{\mu}\,) = \fa^2\fd$ for some integral
  ideal $\fa$. Raising everything to the $h$-th power we get
  $(4^h\mu^h) = (\alpha^2 \delta)$, where  $\fa^h = (\alpha)$ and
  $\fd^h = (\delta)$; moreover we can (and will) choose $\delta$
  in such a way that $F(\sqrt{\mu}\,) = F(\sqrt{\delta}\,)$. 
\end{proof}

The coset $\sep(K/F) = \delta \cdot E_F^2$ is called the separant of
the quadratic extension $K/F$. Separants form a group $\Sep(F)$
with respect to the following multiplication: if
$K_1 = F(\sqrt{\alpha}\,)$ and $K_2 = F(\sqrt{\beta}\,)$ are
quadratic extensions with separants $\delta_1 E_F^2$ and $\delta_2 E_F^2$, let
$K_3 = F(\sqrt{\alpha\beta}\,)$ be the third quadratic subfield of $K_1K_2/F$,
and set
$$ \delta_1 E_F^2 * \delta_2 E_F^2 = \delta_3 E_F^2, $$
where $\delta_3 E_F^2 = \sep(K_3/F)$. Observe that this is the usual
product if and only if the separants are coprime.
Over the field $F = \Q$, a separant of a quadratic
number field is just its discriminant because $E_\Q^2 = 1$.

\medskip\noindent{\bf Example 2.}
Consider $F = \Q(\sqrt{-23}\,)$ and $K = F(\sqrt{\mu}\,)$ for
$\mu = -5+2\sqrt{-23}$. Since $(\mu) = \fp^2\fq$ for the prime
ideals $\fp = (3,1-\sqrt{-23}\,)$ and $\fq = (13, 4+\sqrt{-23}\,)$,
we have $\Ram_f(K/F) = \{\fq\}$ and $\disc K/F = \fq$. Also we
find $\fq^3 = (\delta)$ for $\delta = -37-6\sqrt{-23}$, and
$\sep(K/F) = \delta E_F^2$.
\medskip

Separants that cannot be written as a product of separants are
called irreducible; separants $\delta \cdot E_F^2$ for which
$(\delta)$ is a power of a prime ideal are called prime separants.

\medskip\noindent{\bf Example 3.}
Consider the quadratic extension $K = F(\sqrt{1+2i}\,)$ of $F = \Q(i)$.
It has relative integral basis $\{1, \frac{1 + \sqrt{1+2i}}{1+i}\}$, and
its relative discriminant is therefore the ideal $(1+i)^2(1+2i)$. The
generator $\delta = 2i(1+2i)$ of this ideal has the property that
$K = F(\sqrt{\delta}\,)$, hence it is the separant of this extension:
$\sep(K/F) = 2i(1+2i)$. Since there is no quadratic extension of $F$
with relative discriminant $\fd = (2)$, this separant cannot be factored
into prime separants.

In my early attempts at defining a separant class group I called
$2i$ an ideal separant, defined as the greatest common divisor of
the separants $2i(1+2i)$ and $(2i(3+2i)$. Only later I realized that
this ideal separant may be identified with the nontrivial Dirichlet
character modulo $2$ in $\Z[i]$.
\medskip

The modulus associated to a separant $\sep(K/F)$ is the conductor
of the extension $K/F$, including the infinite places; two separants
are called coprime if they are coprime as ideals, and strongly coprime
if their associated conductors are coprime.

Since the square root of the separant $\sep(K/F)$ generates the quadratic
extension $K/F$, the splitting of primes in $K/F$ is determined by
the quadratic Kronecker symbol $(\delta/\fp)$.

\section{Unique Factorization into Prime Separants}

We now explain how to generalize the prime discriminant factorization
considered by Goldstein and Sunley to fields with trivial separant
class group. Below, $K_\gen^{(2)}$ denotes the genus $2$-class field of $K$
over the field $F$ (with odd class number), i.e., the maximal elementary
abelian $2$-extension of $F$ unramified over $K$; the extension
$K_\gen^{(2)+}$ is defined similarly, but we allow ramification at infinite
primes.

\begin{thm}\label{TM2}
Let $F$ be an algebraic number field. Then the following assertions
are equivalent:
\begin{enumerate}
\item The separant class group of $F$ is trivial: $\SCl(F) = 1$.
\item $F$ is totally real, and its class number in the strict sense is odd.
\item $F$ has odd class number, and for every Dirichlet character $\chi$
  defined modulo $4$ there is a unit $\eps \in E_F$ such that
  $\chi = (\frac{\eps}{\cdot})$.  
\item For each prime ideal $\fp$ with
  odd norm there is a (unique) quadratic extension $K/F$ with
  $\Ram_f(K/F) = \{ \fp \}$.  
\item The class number of $F$ is odd, and every separant in $\Sep(F)$
  can be written uniquely (up to order) as a product of prime separants.
\item For every quadratic extension $K/F$ we have 
  $(K_\gen^{(2)+}:F) = 2^t$, where $t = \# \Ram_f(K/F)$.
\end{enumerate}
\end{thm}

\begin{proof}
  (1) $\iff$ (2) is a consequence of the formula $\# \SCl(F) = 2^{\rho^++s}$,
  where $\rho^+ = \rk \Cl_2(F)$ and $2s$ is the number of complex
  embeddings of $F$.
  
  (1)  $\iff$ (3) is just a reformulation of the content of the
  exact sequence
  $$ \begin{CD}
    1 @>>> \SCl_4(F) @>>> \SCl(F)  @>>> \Cl(F)/\Cl(F)^2 @>>> 1
    \end{CD} $$

  (1)  $\impl$ (4): Let $\fp$ be a prime ideal with odd norm. Since $F$
  has odd class number, $\fp^h = (\pi)$ for some $\pi  \in \cO_F$. Since
  $\SCl(F) = 1$, every element of $M_4/M_4^2$ is represented by a unit;
  thus there is a unit $\eps$ such that $\eps \pi \equiv \xi^2 \bmod 4$.
  In the quadratic extension $K = F(\sqrt{\pi\eps}\,)$, the only ramified
  prime ideal is $\fp$.

  (4)  $\impl$ (3): We first show that the class number of $F$ is odd.
  Let $\fp$ be a prime ideal with odd norm in $F$, and let $K/F$ be a quadratic
  extension in which exactly $\fp$ ramifies. Then $K = F(\sqrt{\alpha}\,)$
  with $(\alpha) = \fp\fa^2$ for some ideal $\fa$. This implies that
  the ideal class of $\fp$ is a square. Since the $[\fp]$ generate the class
  group, every ideal class is a square. Since the class group is finite,
  its order must be odd.

  Now let $\chi$ be a Dirichlet character defined modulo $4$.
  By the proof of Prop.~\ref{PD4}, $\chi$ is in the image of $M_4/M_4^2$,
  i.e. there is a Kronecker character $\kappa = (\frac{\alpha}{\cdot})$
  such that $\kappa = \chi \psi$, where  $\psi$ is defined modulo an odd ideal.
  By assumption, $\psi$ is a Kronecker character; but then so is $\chi$.
  This implies the claim.
  
  (1)  $\impl$ (5): Let $\delta E_F^2$ be a separant. The Kronecker
  character $\kappa_\delta$ can be written uniquely as a product of
  primitive quadratic Dirichlet characters with prime power conductor:
  $\kappa_\delta = \chi_1 \cdots \chi_t$. Since $\Sep(F) = 1$, each
  factor $\chi_j$ is a Kronecker character $(\delta_j/\,\cdot\,)$ for
  separants $\delta_j$. Since the conductors of these Kronecker characters
  are coprime prime ideal powers we must have
  $\delta E_F^2 = \delta_1 \cdots \delta_t E_F^2$.

  We claim that this factorization is unique. To this end, assume that
  there  are two factorizations
  $\delta E_F^2 = \delta_1 \cdots \delta_t E_F^2
               = \delta_1' \cdots \delta_t' E_F^2$
  of $\delta$ into prime separants.

  Since $K_\gen^{(2)+} = F(\sqrt{\delta_1}, \ldots, \sqrt{\delta_t}\,)$
  is the maximal elementary abelian $2$-extension of $F$ unramified over $K$
  at all finite primes, we must have $\sqrt{\delta_j'} \in K_\gen^+$.
  But since $\delta_j'$ and the $\delta_j$ are prime separants, this is
  only possible if $\delta_j' E_F^2$ is among the  $\delta_j E_F^2$.
  
  (5) $\impl$ (6): Let $K/F$ be a quadratic extension with separant $\delta$.
  By assumption we can write
  $\delta E_F^2 = \delta_1 \cdots \delta_t E_F^2$ as a
  product of prime discriminants that are coprime except for infinite prime
  factors. But then $L = F(\sqrt{\delta_1}, \ldots, \sqrt{\delta_t}\,)/K$
  is unramified at all finite primes, hence equal to the genus field in
  the strict sense.

  (6) $\impl$ (4) is Prop.~\ref{Pgen}. 
\end{proof}

Here is the result we have used in the proof:

\begin{prop}\label{Pgen}
  Let $F$ be a field with odd class number in the strict sense. If,
  for every quadratic extension $K/F$, we have $(K_\gen^{(2)+}:F) = 2^t$,
  where $t = \# \Ram_f(K/F)$, then for each odd prime ideal $\fp$ there
  is a unique quadratic extension $k/F$ in which $\fp$ is the only
  (finite) ramified prime ideal.
\end{prop}

\begin{proof}
  Let $\fp$ be a prime ideal in $F$ with odd norm. Let $K/F$ be a
  quadratic extension in which $\fp$ is ramified (choose an ideal
  $\fa$ in the inverse class of $[\fp]$ coprime to $\fp$ and write
  $\fp\fa = (\alpha)$; then $K = F(\sqrt{\alpha}\,)$ is such an
  extension). Then $K_\gen^{(2)+}/$ is an elementary abelian extension
  in which every ramified prime ideal has ramification index $2$.
  Let $\fq_1$, \ldots, $\fq_{t-1}$ denote the ramified prime ideals
  different from $\fq$. Their ramification fields have degree $2^{t-1}$
  over $F$. Thus their intersection is a quadratic
  extension $k/F$ in which the $\fq_j$ are unramified. Since $F$
  admits no unramified extension, we must have $\Ram(k/F) = \{\fp\}$.

  The extension $k$ is necessarily unique; if $k'$ is another such extension,
  the quadratic subextension of $kk'/F$ different from $k$ and $k'$ would
  be unramified at all finite primes, contradicting the fact that $F$ has
  odd class number in the strict sense.
\end{proof}

As another application, consider $k = \Q(\sqrt{130}\,)$ and its genus
class field $k_\gen = \Q(\sqrt{2},\sqrt{5},\sqrt{13}\,)$. 
We have $\SCl(F) = 1$ for $F = \Q(\sqrt{2}\,)$, and the unique
factorization of the separant of $F_1 = F(\sqrt{65}\,)$ is
$65 = 5 \cdot 13$. Since we also have $\SCl(L) = 1$ for
$L = F(\sqrt{5}\,)$, the relative discriminant of $k_\gen/L$ factors into
prime separants. Now $26$ is represented by the form $x^2 - 10y^2$, so
we find $26 = 6^2 - 10$ and $13 = 2 \cdot 3^2 - 5$. Thus
$3\sqrt{2} + \sqrt{5}$ is a factor of $\sep(k_\gen/L)$. Since $\SCl(L) = 1$
there must exist a unit $\eps \in E_L$ such that 
$(3\sqrt{2} + \sqrt{5}\,)\eps \equiv \xi^2 \bmod 4$. We quickly find that 
$\alpha = (3\sqrt{2} + \sqrt{5}\,)(3 + \sqrt{10}\,)(2+\sqrt{5}\,)$ is
$2$-primary (coprime to $(2)$ and congruent to a square modulo $4$), 
hence its square root generates an unramified quadratic extension of
$k_\gen$ whose Galois group over $k$ is the quaternion group of order $8$.  
The minimal polynomial of $\sqrt{\alpha}$ is
$f(x) = x^8 - 180x^6 + 1854x^4 - 2340x^2 + 169$.

\section{Genus theory}

Genus theory of quadratic extensions $k/F$ of number fields with trivial
separant class group can be made as explicit as over $\Q$: the factorization
of $\delta = \sep(k/F)$ into prime separants gives us the genus class field
as well as the genus characters.

\begin{lem}
  Let $F$ be a number field with odd class number in the strict sense.
  Let $K = F(\sqrt{\delta}\,)$ and assume that
  $L = F(\sqrt{\delta_1},\sqrt{\delta_2}\,)$ is a quadratic extension of $K$,  
  were $\delta$, $\delta_1$ and $\delta_2$ are separants. Then $L/K$ is
  unramified at all finite primes if and only if
  $\delta E_F^2 = \delta_1\delta_2 E_F^2$.
\end{lem}

\begin{proof}
  Let $k_1 = F(\sqrt{\delta_1}\,)$, $k_2 = K(\sqrt{\delta_2}\,)$
  and $K = F(\sqrt{\delta_1\delta_2}\,) = F(\sqrt{\delta}\,)$ be the
  three quadratic subextensions of $K/F$. By the conductor-discriminant
  formula we have
  $$ \disc (L/F) = \disc (k/F) \cdot \disc (k_1/F) \cdot \disc (k_2/F). $$
  Since $L/K$ is unramified, we have $\disc(L/F) = \disc(k/F)^2$. This
  implies $\disc (k_1/F) \cdot \disc (k_2/F) = \disc(k/F)$. Since
  separants are $h$-th powers of the relative discriminants, we must
  have $\delta_1 \cdot \delta_2 = \eps \delta$. Since
  $F(\sqrt{\delta_1\delta_2}\,) = F(\sqrt{\delta}\,)$, the unit
  $\eps$ must be a square, and this implies that
  $\delta_1 E^2 \cdot \delta_2 E^2 = \delta E^2$ as claimed.
\end{proof}

Let $F$ be a number field with $\SCl(F) = 1$. For a quadratic extension
$K/F$ with separant $\sep(K/F) = \delta E_F^2$ let
$ \delta E_F^2 = \delta_1 \cdots \delta_t E_F^2$
denote the factorization of $\sep(K/F)$ into prime separants.
We already know that this factorization determines the genus field of $K/F$,
and that the genus field in the strict sense is
$K_\gen^+ = F(\sqrt{\delta_1}, \ldots, \sqrt{\delta_t}\,)$.

Now let $c \in \Cl(K)$ be an ideal class, pick an ideal $\fa \in c$ coprime to
$\delta$. Then $\chi_j(c) = (\frac{\delta_j}{N_{K/F} \fa})$ is a well
defined quadratic character on the class group. The characters $\chi_j$
are called genus characters of $K/F$.

\begin{thm}[Principal Genus theorem]
  Let $F$ be a number field with trivial separant class group. Then
  an ideal class $c \in \Cl^+(K)$ is a square if and only if
  $\chi_j(c) = +1$ for all genus characters $\chi_j$.
\end{thm}

\begin{proof}
  Write $c = [\fp]$ for some prime ideal $\fp$ unramified in $K/F$.
  Clearly $\chi_j(c) = 1$ for all genus characters is equivalent
  to $\fp$ splitting completely in the genus class field.
  Since $F$ has odd class number in the strict sense, $K_\gen^+$ is
  the maximal elementary abelian $2$-extension contained in the
  Hilbert class field in the strict sense. By class field theory,
  $\fp$ splits completely in the maximal elementary abelian
  $2$-extension contained in the  Hilbert class field in the strict sense
  if and only if the ideal class $[\fp]$ is a square in the class
  group $\Cl_2^+(k)$.
\end{proof}

\subsection*{Example.}
Consider the quadratic extension $k = F(\sqrt{93 + 50 \sqrt{2}}\,)$ of
$F = \Q(\sqrt{2}\,)$. The separant
$\delta = \sep(k/F) = 4(93 + 50 \sqrt{2}\,) E_F^2$ has the factorization
$$ 4(93 + 50 \sqrt{2}\,) = -4(-11 - 4\sqrt{2}\,)(7 + 2\sqrt{2}\,) $$
into prime separants (modulo squares of units). Thus 
$$ K_\gen = K\Big(\sqrt{7+2\sqrt{2}}\,\Big) \quad \text{and} \quad
   K_\gen^+ = K\Big(\sqrt{-1}, \ \sqrt{7+2\sqrt{2}}\,\Big), $$
hence $\Cl_2(k)$ has rank $1$ and $\Cl_2^+(k)$ has rank $2$.
{\tt pari} tells us that in fact $\Cl(K) = [2]$ and  $\Cl^+(k) = [2,2]$. 

\section{Hilbert's First Supplementary Law}

Hilbert's version of the first supplementary law in number fields
(see \cite[Theorem 6.5]{LemSel}) 
is the following:

\begin{prop}
  Let $\fa$ be an ideal with odd norm in a number field $F$ with class
  number $h$. Then $\fa^h = (\alpha)$, and we can choose $\alpha$
  in such a way that $\alpha \gg 0$ and $\alpha \equiv \xi^2 \bmod 4$
  if and only if $(\frac{\beta}{\fa}) = +1$ for all elements
  $\beta \in \Sel(F)$.
\end{prop}

Here we will briefly mention a few variants of Hilbert's first
supplementary law in number fields with odd class number (the proofs
are all similar), where we can replace the Selmer group by the unit group.

\begin{prop}
  Let $F$ be a number field with odd class number $h$, and $\fa$
  an ideal with odd norm. Then $\fa^h = (\alpha)$ for some
  $\alpha \in F^\times$, and the following assertions are equivalent:

  \begin{center}
    $$ \begin{array}{c|c}
      \rsp \text{conditions on $\fa$} & \text{conditions on $\alpha$} \\ \hline
      \rsp (E/\fa)   = 1 & \alpha \equiv \xi^2 \bmod 4, \ \alpha \gg 0 \\
      \rsp (E_4/\fa) = 1 & \alpha \gg 0 \\
      \rsp (E^+/\fa) = 1 & \alpha \equiv \xi^2 \bmod 4
    \end{array} $$    
  \end{center}
\end{prop}

If $\SCl(F) = 1$, then $E^+ = E^2$, hence $(E^+/\fa) = +1$ for all ideals
with odd norm, and thus every ideal $\fa^h$ is generated by an element
$\alpha \equiv \xi^2 \bmod 4$ as predicted by our results above.

In quadratic number fields $F = \Q(\sqrt{m}\,)$ with $m \equiv 3 \bmod 4$
we have $-1 \in E_4$ and $(\frac{-1}{\fa}) = (-1)^{(N\fa -1)/2}$,
hence $(\frac{-1}{\alpha}) = +1$ for totally positive elements with norm
$\equiv 1 \bmod 4$. Thus ideals with norm $\equiv 1 \bmod 4$ have a totally
positive generator, and ideals with norm $\equiv 3 \bmod 4$ do not.
In $\Q(\sqrt{3}\,)$, the prime ideal above $13 \equiv 1 \bmod 4$ is
generated by $4 + \sqrt{3} \gg 0$, and the prime ideal with norm $11$ by
$1 + 2\sqrt{3}$. Next $2(2+\sqrt{3}\,) = (1 + \sqrt{3}\,)^2$, hence
$(\frac{\eps}{\fa}) = (\frac{2}{N\fa})$, and $(E/\fp) = +1$ if and
only if $N\fp \equiv 1 \bmod 8$. This implies that the prime ideal
above $13$ does not have a totally positive $2$-primary generator
(we have $(4 + \sqrt{3}\,)(2+\sqrt{3}\,) = 11 + 6\sqrt{3}$, which is
not congruent to a square modulo $4$), but that the prime ideal above
$73$ has, and in fact $11 + 4 \sqrt{3}$ is such a generator.

\subsection*{Final Remarks}

Separants were used in \cite{Lem95} for constructing unramified $2$-extensions
of quadratic extensions $K/F$ of fields $F$ with $\SCl(F) = 1$. The simplest
result was that if $F$ is a number field with $\SCl(F) = 1$ and
$k/F$ a quadratic extension with separant $\delta E_F^2$, then there is
an unramified cyclic quartic extension $K/k$ if and only if $\delta$ admits
a $C_4$-factorization, i.e., if  $\delta E_F^2 = \delta_1 \delta_2 E_F^2$
with $(\delta_1/\fp_2) = (\delta_2/\fp_1)$ for all prime ideals
$\fp_i \mid \delta_j$. In this case, the extension can be constructed
explicitly by solving the equation $\delta_1 x^2 + \delta_2 y^2 = z^2$
in $F^\times$.

I expect that the main results of R\'edei (R\'edei matrix, triple symbol
reciprocity; see Stevenhagen \cite{Stev}), Koch \cite{Koch},
Fr\"ohlich \cite{Fro2} and the existence of governing fields generalize
to all number fields with trivial separant class group.
Kuramoto \cite{Kura} recently has  discussed the R\'edei symbol in
real quadratic fields with class number $1$.

\subsection*{Acknowledgements}

I thank the referees for their careful reading of the manuscript
and numerous very helpful remarks and corrections. I also thank
Chip Snyder for his comments.

\end{document}